\documentclass[a4paper,11pt]{amsart}
\usepackage[english]{babel}
\usepackage[latin1]{inputenc}
\usepackage{amsmath,amsfonts,amssymb,amsthm,amscd,array,stmaryrd,mathrsfs}
\usepackage[all]{xy}
\usepackage{anysize}\marginsize{25mm}{25mm}{30mm}{30mm}
\allowdisplaybreaks[4]
\usepackage{graphicx}
\usepackage{color}              
\usepackage[colorlinks,unicode=false,linkcolor=blue,citecolor=blue]{hyperref}

\numberwithin{equation}{section}
\theoremstyle{plain}
\newtheorem{thm}{Theorem}
\newtheorem{lem}{Lemma}[section]
\newtheorem{cor}[lem]{Corollary}
\newtheorem{prop}[lem]{Proposition}
\theoremstyle{definition}
\newtheorem{defi}[lem]{Definition}
\newtheorem{rem}[lem]{Remark}
\newtheorem{ex}[lem]{Example}

\newcommand{\R}{\mathbb{R}}
\newcommand{\Z}{\mathbb{Z}}
\newcommand{\C}{\mathbb{C}}

\newcommand{\bbH}{\mathbb{H}}

\newcommand{\bbO}{\mathbb{O}}
\newcommand{\bbP}{\mathbb{P}}
\newcommand{\bbS}{\mathbb{S}}

\newcommand{\N}{\mathbb{N}}

\newcommand{\cN}{\mathcal{N}}

\newcommand{\Span}{\mathrm{Span}}

\newcommand{\card}{\textup{card}}

\def \Vec{\operatorname{Vec}}
\def\a{\alpha}
\def\b{\beta}

\def\r{\rho}

\begin{document}
\title[$\Z_2^n$-graded quasialgebras and the Hurwitz problem]
{$\Z_2^n$-graded quasialgebras and the Hurwitz problem on compositions of quadratic forms$^\dag$}
\thanks{\tiny $^\dag$Supported by SRFDP 20130131110001, NSFC 11471186 and 11571199.}
\author[Y.-Q. Hu]{Ya-Qing Hu}
\address{Ya-Qing Hu, School of Mathematics, Shandong University, Jinan 250100, China}
\email{yachinghu@mail.sdu.edu.cn}
\author[H.-L. Huang]{Hua-Lin Huang*}\thanks{*Corresponding author.}
\address{Hua-Lin Huang, School of Mathematical Sciences, Huaqiao University, Quanzhou 362021, China}
\email{hualin.huang@hqu.edu.cn}
\author[C. Zhang]{Chi Zhang}
\address{Chi Zhang, School of Mathematics, Shandong University, Jinan 250100, China}
\email{chizhang@mail.sdu.edu.cn}
\date{}
\begin{abstract}
We introduce a series of $\Z_2^n$-graded quasialgebras $\bbP_n(m)$
which generalizes Clifford algebras, higher octonions, and higher Cayley algebras.
The constructed series of algebras and their minor perturbations are applied to contribute explicit solutions to
the Hurwitz problem on compositions of quadratic forms.
In particular, we provide explicit expressions of the well-known Hurwitz-Radon square identities in a uniform way,
recover the Yuzvinsky-Lam-Smith formulas,
confirm the third family of admissible triples proposed by Yuzvinsky in 1984,
improve the two infinite families of solutions obtained recently by Lenzhen,
Morier-Genoud and Ovsienko, and construct several new infinite families of solutions.
\end{abstract}
\maketitle
\medskip
\thispagestyle{empty}
\tableofcontents
\section{Introduction}
As a special class of algebras in braided tensor categories \cite{M},
group graded quasialgebras were introduced and investigated in \cite{AM} by Albuquerque and Majid.
The familiar twisted group algebras may be naturally viewed as group graded quasialgebras
and this viewpoint provides new derivation of some interesting known results and even helps to generalize them,
see for instance the subsequent work of Albuquerque and Majid on Clifford algebras \cite{AM1}.
More importantly, this viewpoint allows us to treat some mysterious {\it nonassociative} algebras exactly
as if they were effectively associative with a help of the theory of tensor categories (see for example \cite{M1}).
As a marvelous example, the octonions were realized as a $\Z_2^3$-graded quasialgebra in \cite{AM}.
The idea was applied by Morier-Genoud and Ovsienko to
introduce an interesting series of algebras generalizing the octonions \cite{MGO},
which helps to provide new solutions to the Hurwitz problem on compositions of quadratic forms \cite{LMGO}.
\par
Recall that the Hurwitz problem on compositions of quadratic forms \cite{H}
asks for the description of all triples of positive integers $[r,s,N],$
which are {\it admissible} in the sense that there exists a sum of squares formula of the type
\[ (a_1^2+a_2^2+\cdots+a_r^2)(b_1^2+b_2^2+\cdots+b_s^2)=c_1^2+c_2^2+\cdots+c_N^2 \]
where $a=(a_1, a_2, \cdots, a_r)$ and $b=(b_1, b_2, \cdots, b_s)$ are systems of indeterminates
and each $c_k$ is a bilinear form in $a$ and $b$ with coefficients in a field.
This century-old problem is centrally located in mathematics and remains widely open, see \cite{R, S} for a full overview.
The admissible triples of the form $[N,N,N]$ were determined
by Hurwitz \cite{H} in 1898 and the result is his well-known ``1, 2, 4, 8 Theorem".
A more general series of admissible triples of the form $[r, N, N]$ was settled independently
by Hurwitz \cite{H1} in 1918 and by Radon \cite{Rad} in 1922.
The celebrated Hurwitz-Radon Theorem states that $[r, N, N]$ is admissible
if and only if $r \le \rho(N)$ where $\rho$ is the Hurwitz-Radon function defined by
$\rho(N)=8\alpha+2^\beta$ if $N=2^{4\alpha+\beta}(2\gamma+1)$ with $0 \le \beta \le 3$.
The result of Hurwitz and Radon relies essentially on Clifford algebras, see for instance \cite{ABS}.
\par
Afterwards, many more new solutions to the Hurwitz problem have been obtained by various mathematicians,
see \cite{S} and references therein.
In particular, Yuzvinsky proposed in the early 1980s three families of admissible triples of form
\begin{equation}
\label{eq:Yuzvinsky}
[2n+2, 2^n-\varphi(n), 2^n] \quad \hbox{where} \quad
\varphi(n)=\left\{\begin{array}{lll}
             \binom{n  }{  n  /2}, &  n \equiv 0 \mod 4; \\
            2\binom{n-1}{(n-1)/2}, &  n \equiv 1 \mod 4; \\
            4\binom{n-2}{(n-2)/2}, &  n \equiv 2 \mod 4.
           \end{array}\right.
\end{equation}
Yuzvinsky's novel idea relies on orthogonal pairings on the group algebra over $\Z_2^n$, see \cite{Y1, Y2}.
The proof of Yuzvinsky is very concise and contains a number of errors and gaps.
In \cite{LS}, Lam and Smith have carefully corrected and clarified the proof for the cases in which $n \equiv 0,1 \mod 4.$
Hence these solutions are called the Yuzvinsky-Lam-Smith formulas in literatures.
However, the method of Yuzvinsky for the case in which $n \equiv 2 \mod 4$ is shown to be futile in \cite{LS}
and so the solutions of this case have stayed conjectural.
\par
Remarkably, the four normed division algebras connected to the Hurwitz ``1, 2, 4, 8 Theorem",
Clifford algebras connected to the Hurwitz-Radon Theorem, and the algebras connected to
the Yuzvinsky-Lam-Smith formulas are twisted group algebras over $\Z_2^n,$ see \cite{AM, AM1}.
With this insight, Morier-Genoud and Ovsienko recently introduced in \cite{MGO} their series of algebras $\mathbb{O}_n$
which generalizes the algebra of octonions and Clifford algebras simultaneously and
applied them to get explicit expressions of the Hurwitz-Radon square identities with the exception $n\equiv 0\mod 4$.
In their subsequent work \cite{LMGO} with Lenzhen,
the algebras $\mathbb{O}_n$ were applied to provide two infinite families of admissible triples
which are similar to the Yuzvinsky-Lam-Smith formulas.
\par
Note in hindsight that these algebras useful for the Hurwitz problem are $\Z_2^n$-graded quasialgebras. 
So it is tempting to try other similar algebras in the problem. 
In the present paper, we introduce a class of more complicated $\Z_2^n$-graded quasialgebras $\bbP_n(m)$ $(n \ge m \ge 4)$
and apply them to contribute new solutions to the Hurwitz problem.
The twisting function of our series of algebras $\bbP_n(m)$ contains
all the terms of $\bbO_n$ and additional terms of degree $\ge 4.$
Another feature of our algebras $\bbP_n(m)$ is that their twisting functions are linear
with respect to the second argument which is crucial in the searching for multiplicative pairs.
By applying the composition of Euclidean norms on the series of algebras $\bbP_n(4)$,
we provide explicit expressions of the Hurwitz-Radon square identities
in a uniform way without the exception $n\equiv 0\mod 4$ occurred in \cite{MGO}.
By choosing suitable multiplicative pairs in some delicate perturbations of $\bbP_n(4), \ \bbO_n,$ and $\bbP_n(n),$
we are able to recover the Yuzvinsky-Lam-Smith formulas in a simpler way,
and moreover, confirm Yuzvinsky's third conjectural family of admissible triples.
Using the similar idea, we find the duality between
the two basic methods of generating new admissible triples by addition and subtraction
and improve the Lenzhen-Morier-Genoud-Ovsienko admissible triples.
Furthermore, by applying the method of addition on the Yuzvinsky admissible triples,
we obtain three new infinite families as well.
\par
Recall that an admissible triple $[r,s,N]$ is said to be \emph{optimal} 
if there exists no admissible triples of sizes $[r+1,s,N]$, $[r,s+1,N]$ or $[r,s,N-1]$. 
The admissible triples obtained in this paper recover several well-known optimal triples, 
such as $[10,16,28]$ and $[10,22,30]$, and some best known triples, 
$[10,16,28]$, $[10,22,30]$, $[12,44,60]$, $[12,38,58]$, $[12,32,52]$ and $[14,40,64]$, see \cite{S}. 
In general, it is very challenging to determine 
whether a given admissible triple $[r,s,N]$ is optimal when the entries are big. 
So far, we have not obtained any result about the optimality of our new triples. 
\par
An outline of the paper is as follows.
Section \ref{sec:Z_2^n-graded} is devoted to the basic definitions of group graded quasialgebras and multiplicative pairs.
Some interesting $\Z_2^n$-graded quasialgebras are recalled and the series of algebras $\bbP_n(m)$ are introduced.
The uniform and explicit expression of the Hurwitz-Radon square identities is then given in Section \ref{sec:HRSIR}.
In Section \ref{sec:Yuzvinsky}, simpler constructions of Yuzvinsky-Lam-Smith formulas are provided
and Yuzvinsky's third family of admissible triples is verified.
In Section \ref{sec:LMGO}, some improvements of the Lenzhen-Morier-Genoud-Ovsienko formulas are presented.
The duality of two basic methods of generating admissible triples is also mentioned in passing.
Three new infinite series of admissible triples are shown in Section \ref{sec:newSeries}.
\par
Throughout the paper, by $\Z_2$ we mean the cyclic group of order $2$ and
by $\Z_2^n$ the direct product of $n$ copies of $\Z_2;$ we work on the field $\R$ of real numbers for convenience.
It is not hard to see that most of our results hold for any field of characteristic not $2$.
\par
\section{$\Z_2^n$-graded quasialgebras and multiplicative pairs}\label{sec:Z_2^n-graded}
In this section, we recall the notion of group graded quasialgebras,
some interesting examples of $\Z_2^n$-graded quasialgebras
and introduce some new series which will be applied in later sections.
For later applications in constructing admissible triples,
the notion of multiplicative pairs and the multiplicativity criterion are also recalled.
\subsection{Group graded quasialgebras}
Let $G$ be a finite abelian group with unit $e$ and
let $\phi$ be a normalized 3-cocycle on $G$ with coefficients in $\R^*$.
A $G$-graded quasialgebra is a $G$-graded $\R$-vector space $A$,
a product map $A \otimes A \to A$ preserving the total degree and associative in the sense
\[ (a \cdot b) \cdot c = \phi(|a|, |b|, |c|) a \cdot (b \cdot c) \]
for all homogeneous $a, b, c \in A$ and a 3-cocycle $\phi$.
Here and below, $|a|$ denotes the degree of $a$. A $G$-graded quasialgebra is called coboundary if $\phi$ is coboundary.
As pointed out in \cite{AM}, group graded quasialgebras may be defined in a more general setting.
Given $G$ and $\phi$ as above, if we extend $\phi$ linearly to $\R G^{\otimes 3}$,
then we can regard $(\R G, \phi)$ as a coquasi-Hopf algebra.
Let $\Vec_G^\phi$ denote the comodule category of $(\R G, \phi)$
and $\Vec_G^{\phi}$ has a natural structure of tensor category, see \cite{M1}.
Then a $G$-graded quasialgebra is nothing other than an algebra
in the tensor category $\Vec_G^{\phi}$ as defined in \cite{M}.
This viewpoint may allow us to apply Drinfeld's useful idea of gauge equivalence, see \cite{D} and \cite{AM}.
Natural generalization of twisted group algebras provides an interesting class of group graded quasialgebras.
Let $F: G \times G \to \R^*$ be a normalized 2-cochain, that is, $F$ obeys $F(x, e)=1=F(e,x)$ for all $x \in G$.
Denote by $\phi$ the differential of $F$, i.e., \[ \phi(x,y,z)=\frac{F(x,y)F(xy,z)}{F(y,z)F(x,yz)}.\]
Define the (generalized) twisted group algebra $\R_F G$ as follows.
It has the same vector space as the usual group algebra $\R G$ but a twisted product by $F$,
namely \[ x \cdot y = F(x,y)xy, \quad \forall x, y \in G.\]
Clearly, $\R_F G$ is an algebra in $\Vec_G^{\phi}$, or a $G$-graded quasialgebra,
and it is coboundary since $\phi=\partial F$.
\subsection{Examples of $\Z_2^n$-graded quasialgebras}
Remarkably enough, many interesting algebras, such as the normed division algebras and Clifford algebras,
turn out to be coboundary $\Z_2^n$-graded quasialgebras.
Now we recall some examples from \cite{AM, AM1, MGO}.
In the following, we write elements of the group $\Z_2^n$
in the form $x=(x_1, \cdots, x_n)$ with $x_i \in \{0,1\}$
and its multiplication by $+$ and we present the vector space $\R \Z_2^n$
by $\bigoplus_{x \in \Z_2^n} \R u_x$.
Define functions $f_m: \Z_2^n \times \Z_2^n \to \Z_2$ for all $1\le m \le n$ by
\begin{align}
& f_1=\sum_{i} x_iy_i,\quad f_2=\sum_{i<j} x_iy_j,                            \quad
  f_3=\sum\limits_{\substack{{\rm distinct}\ i,j,k \\ i<j}} x_ix_jy_k,        \quad \notag  \\
& f_4=\sum\limits_{\substack{{\rm distinct}\ i,j,k,l\\ i<j<k}} x_ix_jx_ky_l,  \quad \cdots, \\
& f_m=\sum\limits_{\substack{{\rm distinct}\ i_1,i_2,\cdots,i_{m-1},l\\ i_1<\cdots<i_{m-1}}}
                                         x_{i_1}x_{i_2}\cdots x_{i_{m-1}}y_l, \quad \cdots \notag
\end{align}
\begin{ex}
Let $F_{Cl}: \Z_2^n \times \Z_2^n \to \R^*$ be a function defined by
\begin{equation} F_{Cl}(x,y)=(-1)^{f_1(x,y)+f_2(x,y)}. \end{equation}
Then the associated twisted group algebra $\R_{F_{Cl}} \Z_2^n$ is
the well-known real Clifford algebra $Cl_{0,n}$, see \cite{AM1} for detail.
This recovers the algebra of complex numbers $\C$ when $n=1$ and the algebra of quaternions $\bbH$ when $n=2$.
Note that $Cl_{0,n}$ is associative in the usual sense since the function $F_{Cl}$ is a 2-cocycle.
\end{ex}
\begin{ex}
Assume $n \ge 3$. Define the function $F_\bbO: \Z_2^n \times \Z_2^n \to \R^*$ by
\begin{equation} F_\bbO(x,y)=(-1)^{f_1(x,y)+f_2(x,y)+f_3(x,y)}.\end{equation}
Then the twisted group algebra $\R_{F_\bbO} \Z_2^n$ is the algebra of higher octonions $\bbO_n$
introduced in \cite{MGO} by generalizing the realization of octonions via quasialgebra
(i.e., the case of $n=3$) given in \cite{AM}.
Clearly, the series of algebras $\bbO_n$ are nonassociative in the usual sense as the function $F_\bbO$ is not a 2-cocycle.
\end{ex}
\subsection{The series of algebras $\bbP_n(m)$}
Now assume $n \ge m \ge 4$. Define
\begin{equation}
f_{\bbP(m)}(x,y)=\sum_{i=1}^3 f_i(x,y) + f_m(x,y) \quad \mathrm{and} \quad F_{\bbP(m)}(x,y)=(-1)^{f_{\bbP(m)}(x,y)}.
\end{equation}
Let $\bbP_n(m)$ denote the corresponding twisted group algebra $\R_{F_{\bbP(m)}} \Z_2^n$
and we get a series of nonassociative algebras $\{\bbP_n(m)\}_{n \ge 4}$.
\par
It was shown in \cite{AM} that Cayley-Dickson algebras are all coboundary $\Z_2^n$-graded quasialgebras.
Recall that, starting from $\R$, the Cayley-Dickson construction produces the algebra $\C$ of complex numbers,
the algebra $\bbH$ of quaternions, and the algebra $\bbO$ of octonions consecutively.
The immediate follower of the octonions $\bbO$ is the algebra $\bbS$ of sedenions.
It can be realized as the twisted group algebra $\R_{F_\bbS} \Z_2^4$ with
\[F_\bbS(x,y)=F_{\bbP(4)}(x,y) (-1)^{\sum\limits_{\substack{{\rm distinct}\ i,j,k,4\\ j<k}}x_iy_jy_kx_4}\ .\]
Therefore, the series of algebras $\bbP_n(4)$ may be regarded
as a partial generalization of the algebra $\bbS$ of sedenions. Similarly, $\bbP_n(m)$ are partial generalization of the higher Cayley algebras \cite{AM}.
\subsection{Multiplicative pairs and multiplicativity criterion}\label{subsec:multiplicative}
In this subsection we recall some key notions appeared
in the novel method of Yuzvinsky \cite{Y1, Y2}, see also \cite{LS, LMGO}.
Consider the twisted group algebra $\R_F \Z_2^n$.
For every $a=\sum_{x \in \Z_2^n}a_x u_x \in \R_F \Z_2^n$, we define its Euclidean norm by
\begin{equation}
\label{eq:Euclid}
\cN(a)=\sum_{x\in \Z_2^n}a_x^2.
\end{equation}
With this square norm, it is well-known that ``the law of moduli"
\[\cN(a) \cN(b) = \cN(a \cdot b), \quad \forall a, b \] of $\R, \ \C, \ \bbH, \ \bbO$
provides exactly the ``1,2,4,8"-square identities.
The idea of Yuzvinsky is that in order to obtain (more general) square identities
it suffices to find a pair of subsets $A, B$ of $\Z_2^n$ such that
\begin{equation}
\label{eq:NormProd}
\cN(a) \cN(b) = \cN(a \cdot b), \quad \forall a \in \Span_\R(u_x, \ x \in A), \ b \in \Span_\R(u_y, \ y \in B).
\end{equation}
If this is the case, we say that $(A, B)$ is a \emph{multiplicative pair,} or an \emph{orthogonal pairing}.
Given such a pair $(A, B)$, one may induce from it an admissible triple
$[r,s,N]=[\card{}(A),\,\card{}(B),\,\card{}(A+B)]$ and the explicit square identity is
\begin{equation}
\label{eq:SqId}
\left(\sum_{x \in A} a_x^2\right)\left(\sum_{y \in B} b_y^2\right)=\sum_{z \in A+B} c_z^2
\end{equation}
where $A+B$ is the sumset $\{ a+b \mid a \in A, b \in B\}$ and
\[ c_z=\sum\limits_{\substack{z=x+y\\ (x,y) \in A \times B }} F(x,y)a_xb_y
      =\sum_{(x,x+z) \in A\times B} F(x, x+z)a_x b_{x+z}.\]
\par
In the rest of the paper {\em we always assume that the twisting function $F$ is
of the form $F(x,y)=\left(-1\right)^{f(x,y)}$ where $f$ maps $\Z_2^n \times \Z_2^n$ to $\Z_2$.}
When there is no risk of confusion,
the map $f$ is also called the twisting function of the associated twisted group algebra.
This restriction on $F$ allows us to formulate a criterion of the multiplicativity of a pair of subsets of $\Z_2^n$.
The following proposition is well-known, see for example \cite{LMGO}.
We include it here for completeness.
\begin{prop}
\label{prop:multi}
$(A, B)$ forms a multiplicative pair
if and only if for all $z\in A+B$ and for all $x\neq t$ such that $x+z,\ t+z\in B$, the following equation holds
\begin{equation}
\label{eq:fcond0}
f(x,x+z)+f(x,t+z)+f(t,x+z)+f(t,t+z)=1.
\end{equation}
\end{prop}
\begin{proof}
Let $a=\sum_{x\in A} a_x u_x $ and $b=\sum_{y\in B} b_y u_y$.
Then the product of their Euclidean norms is obviously
$$\cN(a)\,\cN(b)=\sum_{(x, y)\in A\times B}\,a_x^2\,b_y^2
                =\sum\limits_{z\in A+B}\quad \sum_{(x,x+z)\in A\times B} a_x^2 b_{x+z}^2.$$
On the other hand,
\begin{align*}
\cN(a\cdot{}b)&=\sum_{z\in A+B}\left(\sum_{(x,x+z)\in A\times B} (-1)^{f(x,x+z)} a_x b_{x+z}\right)^2\\
              &=\sum_{z\in A+B}\quad \sum\limits_{\substack{(x, x+z)\in A\times B\\ (t,t+z)\in A\times B}} (-1)^{f(x,x+z)+f(t,t+z)} a_x a_t b_{x+z} b_{t+z}.
\end{align*}
Hence the condition (\ref{eq:NormProd}) is fulfilled 
if and only if for all $z\in A+B$ and for all $x\neq t$ such that $x+z,\ t+z\in B$,
the coefficients of $a_x\,a_t\,b_{x+z}\,b_{t+z}$ and $a_x\,a_t\,b_{t+z}\,b_{x+z}$
are exactly opposite, i.e., the equation (\ref{eq:fcond0}) holds.
\end{proof}
\par
This proposition has several equivalent forms, and later on we will need the following one:
\begin{cor}
\label{cor:multieq}
$(A, B)$ forms a multiplicative pair
if and only if for all $x\not=t\in A$ and for all $y\in B$ such that $x+t+y\in B$,
the following equation holds
\begin{equation}
\label{eq:fcond1}
f(x,y)+f(t,y)+f(x,x+t+y)+f(t,x+t+y)=1.
\end{equation}
\end{cor}
\par
For later applications, we also need the following simple fact
in which further restriction is put on the twisting function $f$.
\begin{cor}
\label{cor:linear}
Suppose that $F(x,y)=(-1)^{f(x,y)}$ where $f(x,y)$ is linear in its second argument,
then the equation \eqref{eq:fcond0} becomes
\begin{equation}
\label{eq:compcond}
f(x,x+t)+f(t,x+t)=1, \quad \forall x \ne t \in A.
\end{equation}
\end{cor}
\par
In the rest of the paper, our main mission is to search for multiplicative pairs
in suitable $\Z_2^n$-graded quasialgebras $\mathbb{A}.$
According to the previous discussion, we shall frequently encounter the following function
\begin{equation}\label{eq:comptest}
c_{\mathbb{A}}(x,t) \triangleq f_{\mathbb{A}}(x,x+t)+f_{\mathbb{A}}(t,x+t), \quad \forall x, t \in \Z_2^n
\end{equation}
and we call it the {\em test function} of multiplicativity.
\section{The Hurwitz-Radon Square Identities Revisited}\label{sec:HRSIR}
The aim of this section is to give explicit formulas of the Hurwitz-Radon triples $[\r(2^n), 2^n, 2^n].$
Recall that this was partially achieved by Morier-Genoud and Ovsienko in \cite{MGO}
with a help of the higher octonions $\bbO_n$ only with the exception $n \equiv 0 \mod 4.$
It turns out that by applying the series of algebras $\bbP_n(4)$ instead of $\bbO_n,$
we are able to rule out the exception and come up with the square identities in a uniform manner. The crux is to construct the so-called Hurwitzian sets. Our method of construction in this section is combinatorial and seems of independent interest.
\subsection{The Hurwitzian sets}
Thanks to Subsection \ref{subsec:multiplicative}, in order to obtain the Hurwitz-Radon square identities
it suffices to construct multiplicative pairs of form
$(H,\Z_2^n)$ with $\card{H} = \r(2^n)$ in the algebra $\bbP_n(4).$
We will call $H$ a \emph{Hurwitzian set} in the algebra $\bbP_n(4),$ or simply an H-set,
if $(H,\Z_2^n)$ satisfies this condition.
According to Corollary \ref{cor:linear}, an H-set is a subset $H$ of $\Z_2^n$ such that
\[ c_{\bbP(4)}(x,t) = f_{\bbP(4)}(x,x+t)+f_{\bbP(4)}(t,x+t)=1, \quad \forall x \ne t \in H. \]
\par
Given $x=(x_1, x_2, \dots, x_n) \in \Z_2^n,$ set $o(x)=1+\sum_{i=1}^n x_i 2^{n-i}.$
Define an order on $\Z_2^n$ by
$$x > y \Leftrightarrow o(x) > o(y).$$
It is obvious that this order is nothing other than the lexicographic order
and $o(x)$ indicates the exact position of $x$ in $\Z_2^n.$
With this order, all values of $c_{\bbP(4)}(x,t)$ will form a $2^n$-by-$2^n$ symmetric matrix
and all H-sets are characterized as subsets of $\Z_2^n$
such that the corresponding principal submatrices of size $\r(2^n)$
have zeros in the diagonal and all ones off-diagonal.
Therefore, searching for H-sets is equivalent to searching for such principal submatrices.
\par
By the definitions of $f_{\bbP(4)}$ and $f_\bbO,$ we have
\[ c_{\bbP(4)}(x,t)= [f_\bbO(x, x+t) +f_\bbO(t, x+t)] + [f_4(x, x+t) +f_4(t, x+t)]. \]
Let $\b_4(x,t)=f_4(x,x+t)+f_4(t,x+t).$ Note that
\begin{align*}
f_4(x,x)=\sum\limits_{\substack{{\rm distinct}\ i,j,k,l\\ i<j<k}} x_ix_jx_kx_l=4 \sum\limits_{i<j<k<l} x_ix_jx_kx_l=0.
\end{align*}
Hence, $\b_4(x,t)=f_4(x,t)+f_4(t,x).$
On the other hand, it was observed in \cite{MGO} that
\[ f_\bbO(x, x+t) +f_\bbO(t, x+t)=f_\bbO(x+t,x+t).\]
Let $\alpha_\bbO(x) = f_\bbO(x,x).$
We remark that $\alpha_\bbO(x)$ is called a generating function in \cite{MGO}.
Denote by $|x|$ the \emph{weight} of $x\in \Z_2^n,$
i.e., the number of 1 entries in $x$ written as an $n$-tuple of $0$ and $1$.
The explicit formula of $\a_\bbO(x)$ is as follows:
\begin{align*}
\a_\bbO(x)&=\sum_{1\leq{}i<j<k\leq{}n}x_ix_jx_k+\sum_{1\leq{}i<j\leq{}n}\,x_ix_j+\sum_{1\leq{}i\leq{}n}x_i\\
          &=\binom{|x|}{3}+\binom{|x|}{2}+\binom{|x|}{1}.
\end{align*}
Therefore $\a_\bbO(x)$ depends only on $|x|$ and can be easily proved to be 4-periodic:
\begin{equation}
\label{tab:a_bbO}
\setlength{\extrarowheight}{3pt}
\begin{array}{c|*9{|c}}
|x|     & 1 & 2 & 3 & 4 & 5 & 6 & 7 & 8 &\cdots\\ \hline
\a_\bbO & 1 & 1 & 1 & 0 & 1 & 1 & 1 & 0 &\cdots\\
\end{array}
\end{equation}
\subsection{Explicit expressions of the Hurwitz-Radon square identities}
Now we are ready to present our H-set.
To determine an H-set $H=\{\mathbf{h}\}$ in $\Z_2^n,$ it suffices to determine the subset
$o(H)=\{o(\mathbf{h})\}$ in $\{1, 2, 3, \dots, 2^n\}.$
We observe an interesting method to find the latter by the following function.
\par
\begin{defi}\label{def:hfunction}
Define the function $H: \N \to \N,$ called H-function in the following, by
\begin{equation}
\label{eq:H-function}
\setlength{\extrarowheight}{3pt}
\begin{array}{c|*{9}{|c}}
n    & 1\leq n\leq 9 & 10 & 11 & 12 & 13 & 14 & 15  & 16  & n\geq 17  \\ \hline
H(n) &      n-1      & 16 & 32 & 56 & 64 & 88 & 104 & 112 & 2^4 H(n-8)\\
\end{array}
\end{equation}
\end{defi}
Now let $H \subset \Z_2^n$ be the subset such that $o(H)=\{H(i)+1 \mid i=1,\cdots,\r(2^n)\}.$
More precisely, $H=\{\mathbf{h}_i \mid o(\mathbf{h}_i)= H(i)+1, \ 1 \le i \le \r(2^n) \}.$
Recall that an arbitrary $\mathbf{h} \in \Z_2^n$ is obtained by the 2-adic digital expression of $o(\mathbf{h})-1.$
It turns out that such a set $H$ is an H-set.
\par
Before stating the main theorem, we include the following example of H-sets to
offer some flavor to the reader and also for later applications.
\begin{table}[!hbtp]
\setlength{\extrarowheight}{2pt}
\renewcommand{\arraystretch}{1.2}
\caption{$n=7$ and $\r(2^7)=16$}
\label{tab:16h-vectors}
\begin{tabular}{c@{=}c|c@{=}c||c@{=}c|c@{=}c}  \hline
$\mathbf{h}_1$ & $(0,0,0,0,0,0,0)$ & $H(1)$ & $0$ & $\mathbf{h}_9   $ & $(0,0,0,1,0,0,0)$ & $H(9) $& $8  $ \\ \hline
$\mathbf{h}_2$ & $(0,0,0,0,0,0,1)$ & $H(2)$ & $1$ & $\mathbf{h}_{10}$ & $(0,0,1,0,0,0,0)$ & $H(10)$& $16 $ \\ \hline
$\mathbf{h}_3$ & $(0,0,0,0,0,1,0)$ & $H(3)$ & $2$ & $\mathbf{h}_{11}$ & $(0,1,0,0,0,0,0)$ & $H(11)$& $32 $ \\ \hline
$\mathbf{h}_4$ & $(0,0,0,0,0,1,1)$ & $H(4)$ & $3$ & $\mathbf{h}_{12}$ & $(0,1,1,1,0,0,0)$ & $H(12)$& $56 $ \\ \hline
$\mathbf{h}_5$ & $(0,0,0,0,1,0,0)$ & $H(5)$ & $4$ & $\mathbf{h}_{13}$ & $(1,0,0,0,0,0,0)$ & $H(13)$& $64 $ \\ \hline
$\mathbf{h}_6$ & $(0,0,0,0,1,0,1)$ & $H(6)$ & $5$ & $\mathbf{h}_{14}$ & $(1,0,1,1,0,0,0)$ & $H(14)$& $88 $ \\ \hline
$\mathbf{h}_7$ & $(0,0,0,0,1,1,0)$ & $H(7)$ & $6$ & $\mathbf{h}_{15}$ & $(1,1,0,1,0,0,0)$ & $H(15)$& $104$ \\ \hline
$\mathbf{h}_8$ & $(0,0,0,0,1,1,1)$ & $H(8)$ & $7$ & $\mathbf{h}_{16}$ & $(1,1,1,0,0,0,0)$ & $H(16)$& $112$ \\ \hline
\end{tabular}
\end{table}
\par
\begin{thm}
\label{thm:unified}
Let $n$ be a positive integer, $N=2^n$ and $r=\r(N).$ Then the set
$H=\{\mathbf{h_i}\}_{i=1}^{r} \subset \Z_2^n$ as defined above is a Hurwitzian set in $\bbP_n(4).$
\end{thm}
\begin{proof}
First of all, we observe from Table \ref{tab:16h-vectors} that
the weight $|\mathbf{h}_i|$ of each $\mathbf{h}_i$ is at most 3, for all $1 \le i \le 16.$
For any $i \ge 17,$ then there exists a unique integer $k$ such that $9 \le i-8k \le 16.$
Then by the definition of the H-function and the construction of $H,$
it is not hard to see that $|\mathbf{h}_i|=|\mathbf{h}_{i-8k}|.$
Therefore, $|\mathbf{h}| \in \{0,1,2,3\}, \ \forall \mathbf{h} \in H.$
It follows that $|x+y| \in \{1,2,3,4,5,6\}$ for all $x \ne y \in H.$
\par
Next note that, for our purpose it is enough to verify that for all $x \ne y \in H,$
$$c_{\bbP(4)}(x,y)=\a_\bbO(x+y)+\b_4(x,y)=\a_\bbO(x+y)
+\sum\limits_{\substack{{\rm distinct}\ i,j,k,l\\ i<j<k}}(x_ix_jx_ky_l+y_iy_jy_kx_l)=1.$$
In the following we carry out the verification case by case according to $|x+y|.$
\par
\begin{enumerate}
\item $|x+y|=1.$
In this case, according to the symmetry of $c_{\bbP(4)}(x,y)$ we may assume without loss of generality that
$0=x_{i_1} \ne y_{i_1}=1$ and $x_j = y_j, \ \forall j \ne i_1.$
Moreover, there exist at most two other indices $i_2$ and $i_3$
such that $x_{i_2}=x_{i_3}=1=y_{i_2}=y_{i_3}$ as $|x|,|y| \in \{0,1,2,3\}.$
Then we have the following table which includes all the indices $i$ such that $x_i$ or $y_i$ is $1:$
\begin{equation*}
\begin{array}{c|ccc}
  & i_1 & i_2 & i_3 \\ \hline
x & 0   & a_1 & a_2 \\ \hline
y & 1   & a_1 & a_2 \\
\end{array}
\end{equation*}
Note that in the above table $a_1$ or $a_2$ may be $0.$
Now by direct computation, we have $\a_\bbO(x+y)=1$ by \eqref{tab:a_bbO} and
$\b_4(x,y)=0$ due to the fact that $|x| \le 2, \ |y| \le 3$ and $x_i=0$ if $i \notin \{i_2, i_3\}.$
\item $|x+y|=2.$
By a similar analysis, we have the tables of information for $x$ and $y:$
\begin{equation*}
\begin{array}{c|ccc}
  & i_1 & i_2 & i_3 \\ \hline
x & 0   & 0   & a_1 \\ \hline
y & 1   & 1   & a_1 \\
\end{array}\quad\textrm{or}\quad
\begin{array}{c|cccc}
  & i_1 & i_2 & i_3 & i_4 \\ \hline
x & 0   & 1   & a_1 & a_2 \\ \hline
y & 1   & 0   & a_1 & a_2 \\
\end{array}
\end{equation*}
It is not hard to see that $\a_\bbO(x+y)=1$ and $\b_4(x,y)=0$.
\item $|x+y|=3.$
As before, the tables of information for $x$ and $y$ are as follows:
\begin{equation*}
\begin{array}{c|ccc}
  & i_1 & i_2 & i_3 \\ \hline
x & 0   & 0   & 0   \\ \hline
y & 1   & 1   & 1   \\
\end{array}\quad\textrm{or}\quad
\begin{array}{c|cccc}
  & i_1 & i_2 & i_3 & i_4 \\ \hline
x & 0   & 0   & 1   & a_1 \\ \hline
y & 1   & 1   & 0   & a_1 \\
\end{array}
\end{equation*}
In the first case, it is obvious that $\a_\bbO(x+y)=1$ and $\b_4(x,y)=0.$
For the second, when $a_1=0,$ again one has $\a_\bbO(x+y)=1$ and $\b_4(x,y)=0.$
If $a_1$ were $1,$ then $|x|=2$ and $x \in \{ \mathbf{h}_4,\mathbf{h}_6,\mathbf{h}_7 \}$
as we can read from Table \ref{tab:16h-vectors} and
deduce from Definition \ref{def:hfunction} of the H-function $H(n)$ that $|\mathbf{h}_i| \in \{1,3\}$ when $i \ge 9.$
However, in this case we can not find a $y \in H$ such that it matches the second table.
\item $|x+y|=4.$
Similar to case (3). The tables of information for $x$ and $y$ are as follows:
\begin{equation*}
\begin{array}{c|cccc}
  & i_1 & i_2 & i_3 & i_4 \\ \hline
x & 0   & 1   & 1   &   1 \\ \hline
y & 1   & 0   & 0   &   0 \\
\end{array}\quad\textrm{or}\quad
\begin{array}{c|ccccc}
  & i_1 & i_2 & i_3 & i_4 & i_5 \\ \hline
x & 0   & 0   & 1   & 1   & a_1 \\ \hline
y & 1   & 1   & 0   & 0   & a_1 \\
\end{array}
\end{equation*}
In the first case, it is obvious that $\a_\bbO(x+y)=0$ and $\b_4(x,y)=y_{i_1} x_{i_2} x_{i_3} x_{i_4}=1.$
For the second case if $a_1$ were $0,$ then $|x|=|y|=2$ and $x,y \in \{ \mathbf{h}_4,\mathbf{h}_6,\mathbf{h}_7\},$
so one has $|x+y|=2$, which contradicts our assumption;
if $a_1$ were $1,$ then $x, y\in\{ \mathbf{h} \in H \mid |\mathbf{h}|=3\}.$
However, by Table \ref{tab:16h-vectors} and Definition \ref{def:hfunction}, one sees the fact that $|x+y| \in \{2,6\},$
which is also a contradiction.
\item $|x+y|=5.$
The table of information for $x$ and $y$ is of the form:
\begin{equation*}
\begin{array}{c|ccccc}
  & i_1 & i_2 & i_3 & i_4 & i_5 \\ \hline
x & 0   & 0   & 1   & 1   &   1 \\ \hline
y & 1   & 1   & 0   & 0   &   0 \\
\end{array}
\end{equation*}
Hence in this case one has $\a_\bbO(x+y)=1$ and
$\b_4(x,y)=(y_{i_1}+y_{i_2}) x_{i_3} x_{i_4} x_{i_5}=0.$
\item $|x+y|=6.$
The table of information for $x$ and $y$ must be like:
\begin{equation*}
\begin{array}{c|cccccc}
  & i_1 & i_2 & i_3 & i_4 & i_5 & i_6 \\ \hline
x & 0   & 0   & 0   & 1   & 1   & 1   \\ \hline
y & 1   & 1   & 1   & 0   & 0   & 0   \\
\end{array}
\end{equation*}
Again, by direct calculation one has $\a_\bbO(x+y)=1$ and $\b_4(x,y)=3+3=0.$
\end{enumerate}
In summary, $c_{\bbP(4)}(x,y)=1, \ \forall x \ne y \in H$ always holds, then the proof is complete.
\end{proof}
\begin{cor}
\label{cor:unified}
Keep all the assumptions of Theorem \ref{thm:unified}.
Then we have the following explicit expressions of the Hurwitz-Radon square identities
$$\left(\sum_{x \in H} a_x^2\right) \left(\sum_{y\in \Z_2^n} b_y^2\right)
=\sum_{z\in \Z_2^n} c_z^2, $$ where
$$c_z=\sum_{x\in H} (-1)^{f_{\bbP(4)}(x,x+z)} a_x b_{x+z}.$$
\end{cor}
\section{The Yuzvinsky Admissible Triples}\label{sec:Yuzvinsky}
In this section, by using various types of $\Z_2^n$-graded quasialgebras,
we recover the \emph{Yuzvinsky-Lam-Smith Formulas} in a simpler way,
and moreover we confirm the third family of admissible triples
$$\left[2n+2, 2^n-4\binom{n-2}{\dfrac{n-2}{2}}, 2^n\right],\quad n \ge 4\ \hbox{and}\ n\equiv 2\mod 4$$
foreseen by Yuzvinsky in 1984. Throughout this section,
$n$ is always assumed to be $\ge 4$ (in order to make sure that $2n+2<2^n$).
\subsection{Basic ideas}
The three families of admissible triples proposed in \cite{Y2} by Yuzvinsky are
in the neighbourhood of the Hurwitz-Radon triples.
They can be constructed via the latter by adding $1$ or $2$ to the first component and subtracting the second accordingly.
Thus, to obtain such kind of triples via multiplicative pairs,
we should expand the Hurwitzian sets and detect their companions
with a help of the test functions in suitable $\Z_2^n$-graded quasialgebras.
\par
The procedure we adopt goes as follows.
Take some $\Z_2^n$-graded quasialgebra $\mathbb{A}$ with twisting function $f$
in which Hurwitzian sets are known to exist, for examples $\bbO_n$ and $\bbP_n(4),$ and choose a Hurwitzian set $H.$
Then add one or two appropriate elements to $H,$ say $H \bigcup\{z_1\}$ or $H \bigcup\{z_1,z_2\},$
and detect its companion set $B.$ This relies essentially on the test function.
In order to make the cardinality $\card B$ relatively big and keep the computation under control,
we shall make delicate manipulation on the twisting function.
More precisely, we perturb the twisting function $f$ to $f'$ such that $f'|_{H\times\Z_2^n}=f.$
Denote $\delta_i(y)=f'(z_i,y)+f(z_i,y), \ \forall y\in \Z_2^n$ for $i=1,2$,
and call them the \emph{perturbation functions}.
The crux of our method lies in adding to $H$ and detecting $B$ by the symmetry in $\Z_2^n,$
and furthermore in manipulating the twisting functions and choosing the added elements in a consistent way.
\subsection{Another Hurwitzian set in $\bbP_n(4)$}\label{subsec:anotherhset}
For our purpose of this section and the next, we need another Hurwitzian set in $\bbP_n(4)$
which is different from the one given in Section \ref{sec:HRSIR}.
\par
To describe the elements in $\Z_2^n$, in the rest of the paper we use the following notations:
\begin{equation*}
\begin{array}{rcl}
e_0            & \triangleq & (0,0,\ldots,0),            \\
\overline{e_0} & \triangleq & (1,1,\ldots,1),            \\
e_i            & \triangleq & (0,\ldots,0,1,0,\ldots,0),
\textrm{ where 1 only occurs at the $i$-th position,}    \\
\overline{e_i}&\triangleq& \overline{e_0}+e_i=(1,\ldots,1,0,1,\ldots,1),
\textrm{ where 0 only occurs at the $i$-th position,}
\end{array}
\end{equation*} for all $1\leq i\leq n$.
For any $B \subseteq \Z_2^n,$ let $\overline{B}=\overline{e_0}+B=\{\overline{e_0}+y \mid y \in B \}.$
We then introduce the following subset $H$ of $\Z_2^n$:
\begin{equation}
\label{eq:anotherhset}
H=\left\{
\begin{array}{ll}
\{e_0, e_i, \overline{e_i}                 \mid 1\leq i\leq n\}, & \textrm{ for } n\equiv 0 \mod 4.\\
\{e_1+e_i, e_i                             \mid 1\leq i\leq n\}, & \textrm{ for } n\equiv 1 \mod 4,\\
\{e_1+e_i, e_i                             \mid 1\leq i\leq n\}, & \textrm{ for } n\equiv 2 \mod 4,\\
\{e_0, \overline{e_0}, e_i, \overline{e_i} \mid 1\leq i\leq n\}, & \textrm{ for } n\equiv 3 \mod 4.\\
\end{array}\right.
\end{equation}
In each case, the subset $H$ contains exactly $\rho(2^n)$ elements.
\begin{thm}\label{thm:anotherhset}
$H$ is a Hurwitzian set in $\bbP_n(4).$
\end{thm}
\begin{proof}
Similar to Theorem \ref{thm:unified}, we need to calculate the test function $c_{\bbP_n(4)}(x,y)=\a_\bbO(x+y)+\b_4(x,y)$
and verify that $\a_\bbO(x+y)$ and $\b_4(x,y)$ always have different values for any two distinct elements $x, y\in H.$
If $(x,y)$ is one of the following
$$(e_0,\overline{e_0}),\ (e_0,e_i),\ (e_0,\overline{e_i}),\ (e_i,e_j),\ (e_i,e_1+e_j),\ (e_1+e_i,e_1+e_j),
\quad 1\leq i, j\leq n, $$
then it is obvious that $\b_4(x,y)=0$ and $\a_\bbO(x+y)=1$.
And the following direct case by case calculations
\begin{align*}
c_{\bbP_n(4)}(e_i,\overline{e_i})&=\a_\bbO(\overline{e_0})+\binom{n-1}{3}=1,
             &\textrm{ while } n\equiv 0,3\mod 4;    \\
c_{\bbP_n(4)}(e_i,\overline{e_j})&=\a_\bbO(\overline{e_0}+e_i+e_j)+\binom{n-2}{3}=1,
             &\textrm{ while } n\equiv 0,3\mod 4;    \\
c_{\bbP_n(4)}(e_i,\overline{e_0})&=\a_\bbO(\overline{e_0}+e_i)+\binom{n-1}{3}=1,
             &\textrm{ while } n\equiv 0,3\mod 4;    \\
c_{\bbP_n(4)}(\overline{e_0},\overline{e_j})&=\a_\bbO(e_j)+n\binom{n-1}{3}+(n-1)\binom{n-2}{3}=1,
             &\textrm{ while } n\equiv 3\quad\mod 4; \\
c_{\bbP_n(4)}(\overline{e_i},\overline{e_j})&=\a_\bbO(e_i+e_j)+2\binom{n-1}{3}+2(n-2)\binom{n-2}{3}=1,
             &\textrm{ while } n\equiv 0,3\mod 4 \
\end{align*}
will complete the proof.
\end{proof}
\par
Now we are ready to construct the three families of Yuzvinsky admissible triples.
The construction is split into three cases according to the residue of $n$ modulo $4.$
\subsection{The case in which $n \equiv 0 \mod 4$}\label{subsec:Yuzvinsky0}
In this case, let $f=f_{\bbP(4)}$ and $H=\{e_0, e_i, \overline{e_i} \mid 1\leq i\leq n\}$.
We need to add an element $z$ to $H$ and perturb the twisting function $f$ to $f'$ by $\delta.$
According to Corollary \ref{cor:multieq}, one easily obtains the following lemma:
\begin{lem}
$(H\bigcup\{z\}, B)$ forms a multiplicative pair with respect to $f'$
if and only if for all $h\in H$ and for all $y \in B$ such that $h+z+y\in B$,
the following equation
\begin{equation*}
f'(z,y)+f'(z,h+z+y)=f(h,h+z)+1,
\end{equation*}
or equivalently
\begin{equation}\label{eq:delta0}
\delta(y)+\delta(h+z+y)=f(h,h+z)+f(z,h+z)+1
\end{equation} holds.
\end{lem}
\par
We observe that if we choose $z=\overline{e_0},$
then $H\bigcup\{z\}$ has a nice symmetry and the equation \eqref{eq:delta0} is easier to handle.
In this situation, one has
\begin{lem}\label{lem:3delta0}
$(H\bigcup\{\overline{e_0}\}, B)$ is a multiplicative pair
if and only if for all $i=1,\cdots,n$ the following equations hold:
\begin{align}
\label{eq:3delta01}
& \delta(y)+\delta(\overline{y})=1,      \textrm{ whenever } y, \overline{y} \in B; \\
\label{eq:3delta02}
& \delta(y)+\delta(e_i+\overline{y})=1,  \textrm{ whenever } y, e_i+\overline{y} \in B; \\
\label{eq:3delta03}
& \delta(y)+\delta(e_i+y)=0,             \textrm{ whenever } y, e_i+y \in B.
\end{align}
\end{lem}
\begin{proof}
Let $h$ take turns to be $e_0$, $e_i$ and $\overline{e_i}$.
Direct calculations of the right hand side of equation \eqref{eq:delta0} yield the above equations.
\end{proof}
\par
Next we shall choose a suitable companion set $B$ for $H\bigcup\{\overline{e_0}\}$
and a proper perturbation function $\delta.$ By \eqref{eq:3delta01},
we should assume that $B$ is symmetric with respect to $\overline{e_0},$
i.e., there exists a subset $B_1 \subset B$ such that $B= B_1 \biguplus \overline{e_0}+B_1.$
Here and below, $\biguplus$ stands for disjoint union.
By the symmetry of $\Z_2^n$ with respect to the weight of its elements, we may further assume that
$B_1$ is a subset of $\{y\in \Z_2^n \mid |y|\leq \frac{n}{2}\}.$
Accordingly, $\overline{e_0}+B_1$ is a subset of $\{y\in \Z_2^n \mid |y|\geq \frac{n}{2}\}$.
In addition, note that equations \eqref{eq:3delta02} and \eqref{eq:3delta03} are symmetric
in the view of \eqref{eq:3delta01}, so we only need to solve the three equations for $y\in B_1$.
\par
After these preparations, now we are able to state the first Yuzvinsky-Lam-Smith formula.
\begin{thm}
\label{thm:YLS0}
There exists a series of admissible triples of sizes
$$\left[ 2n+2, 2^n-\binom{n}{n/2}, 2^n\right], \quad\textrm{ for } n\geq 4 \textrm{ and } n\equiv 0 \mod 4. $$
\end{thm}
\begin{proof}
Let $B_1=\{y\in \Z_2^n \mid |y|<n/2\}$ and let
$$\delta(y)=\left\{\begin{array}{ll}
                     1, & \textrm{ if } |y|\geq n/2+1; \\
                     0, & \textrm{ otherwise. }
                    \end{array}\right.$$
Then $B=\{y\in \Z_2^n \mid |y|\ne n/2\}, \ H\bigcup\{\overline{e_0}\} + B =\Z_2^n$ and $\card B = 2^n-\binom{n}{n/2}.$
For any $y\in B_1$, we have $|y|\leq n/2-1,\ |\overline{y}|\geq n/2+1,$
and thus $\delta(y)=0,\ \delta(\overline{y})=1$.
Also, $|e_i+\overline{y}|> n/2$ when $e_i+\overline{y} \in B$ and $|e_i+y| < n/2$ when $e_i+y \in B$
hold for any $i=1,\cdots,n$.
Clearly, for all $y\in B_1$ we have $\delta(e_i+\overline{y})=1,\ \delta(e_i+y)=0.$
Now the theorem follows by Lemma \ref{lem:3delta0} and equation \eqref{eq:SqId}.
\end{proof}
\subsection{The case in which $n\equiv 1\mod 4$}\label{subsec:Yuzvinsky1}
The method is similar to the previous case.
In this case, $\r(2^n)=2n$ and we need to add two elements into the chosen H-set.
We observe that there is a symmetric H-set $H =\{e_i, \overline{e_i}, 1\leq i\leq n\}$
in the higher octonions $\bbO_n,$ see \cite{MGO}.
It is tempting to expand $H$ by $\{e_0,\overline{e_0}\}$ according to the obvious symmetry.
\par
For brevity, denote $f_\bbO$ by $f$ and perturb it by $\delta_1$ and $\delta_2$ as follows:
\begin{equation*}
f'(x,y)=\left\{\begin{array}{ll}
               f(x,y)+\delta_1(y), & \textrm{ if } x = e_0;            \\
               f(x,y)+\delta_2(y), & \textrm{ if } x = \overline{e_0}; \\
               f(x,y),             & \textrm{ otherwise. }
               \end{array}\right.
\end{equation*}
\par
Similarly, we have the following lemma for the companion set $B$ of $H\bigcup\{e_0,\overline{e_0}\}.$
\begin{lem}
$(H\bigcup\{e_0,\overline{e_0}\}, B)$ forms a multiplicative pair with respect to $f'$
if and only if the following equations hold:
\begin{align*}
& f'(e_0,y)+f'(e_0,y+h)=f(h,h)+1, \quad
  \forall h\in H, \textrm{ whenever } y, y+h\in B; \\
& f'(\overline{e_0},y)+f'(\overline{e_0},h+\overline{y})=f(h,\overline{h})+1, \quad
  \forall h\in H, \textrm{ whenever } y, h+\overline{y}\in B; \\
& f'(e_0,y)+f'(e_0,\overline{y})+f'(\overline{e_0},y)+f'(\overline{e_0},\overline{y})=1, \quad
  \textrm{ whenever } y, \overline{y}\in B,
\end{align*}
or equivalently
\begin{align}
\label{eq:delta11}
& \delta_1(y)+\delta_1(y+h)=f(h,h)+1,\quad
  \forall h\in H, \textrm{ whenever } y, y+h\in B; \\
\label{eq:delta12}
& \delta_2(y)+\delta_2(h+\overline{y})=f(h,\overline{h})+f(\overline{e_0},\overline{h})+1, \quad
  \forall h\in H, \textrm{ whenever } y, h+\overline{y}\in B; \\
\label{eq:delta13}
& \delta_1(y)+\delta_1(\overline{y})+\delta_2(y)+\delta_2(\overline{y})=0, \quad
  \textrm{ whenever } y, \overline{y}\in B.
\end{align}
\end{lem}
\par
Now if we set $\delta=\delta_1=\delta_2$, then equation \eqref{eq:delta13} holds automatically.
In this situation, the previous lemma is reduced to:
\begin{lem}
$(H\bigcup\{e_0,\overline{e_0}\}, B)$ is a multiplicative pair
if and only if for all $i=1,\cdots,n$ the following equations hold:
\begin{align}
\label{eq:2delta11}
& \delta(y)+\delta(y+e_i)=0,\quad \textrm{ whenever } y, y+e_i\in B; \\
\label{eq:2delta12}
& \delta(y)+\delta(y+\overline{e_i})=1, \quad \textrm{ whenever } y, y+\overline{e_i}\in B.
\end{align}
\end{lem}
\begin{proof}
By direct calculations.
\end{proof}
\par
Now we put some assumptions on $B$ again due to the obvious symmetry of $\Z_2^n.$
Let $B=B_1 \biguplus \overline{e_0}+B_1$ in which $B_1$ is a subset of $\{y\in \Z_2^n \mid |y|\geq \frac{n+1}{2}\}$.
Meanwhile, $\overline{e_0}+B_1$ is a subset of $\{y\in \Z_2^n \mid |y|\leq \frac{n-1}{2}\}$.
It follows that for any $y\in B$, $e_i+y \in B$ if and only if
$\overline{e_i+y} \in B$.
Thus equation \eqref{eq:2delta12} can be replaced by
\begin{equation}
\label{eq:2delta13}
\delta(e_i+y)+\delta(\overline{e_i+y})=1, \quad \textrm{ whenever } y, e_i+y, \overline{e_i+y}\in B.
\end{equation}
Again, we assume the symmetric condition $\delta(y)+\delta(\overline{y})=1, \ \forall y \in B$
on the perturbation function. Then equation \eqref{eq:2delta13} will hold automatically.
\par
The following is the second Yuzvinsky-Lam-Smith formula.
\begin{thm}
\label{thm:YLS1}
There exists a series of admissible triples of sizes
\[\left[ 2n+2, 2^n-2\binom{n-1}{\dfrac{n-1}{2}}, 2^n\right],
\quad\textrm{ for } n\geq 4 \textrm{ and } n\equiv 1 \mod 4.\]
\end{thm}
\begin{proof}
For any $y = (y_1, \dots, y_n) \in \Z_2^n,$
define $$\delta(y)=\left\{\begin{array}{ll}
                      1, & \textrm{ if } |y|\geq \frac{n+1}{2}+1;               \\
                      1, & \textrm{ if } |y|=\frac{n+1}{2} \textrm{ and } y_1=0;\\
                      0, & \textrm{ otherwise. }
                   \end{array}\right.$$
Let $$B_1=\{y\in \Z_2^n\mid |y|\geq\frac{n+1}{2}+1\}\biguplus \{y\in\Z_2^n\mid |y|=\frac{n+1}{2}\textrm{ and } y_1=0\}.$$
Then we have $\overline{e_0}+B_1=\{y\in \Z_2^n \mid |y|\leq \frac{n-1}{2}-1\}
                       \biguplus \{y\in \Z_2^n \mid |y|=\frac{n-1}{2}\textrm{ and } y_1=1\}$ and $\delta(B_1)=1$.
It is not hard to see that
\[\card B = 2^n-2\binom{n-1}{\dfrac{n-1}{2}} \textrm{ and } H\bigcup\{e_0, \overline{e_0}\} + B=\Z_2^n.\]
To prove the theorem, it is enough to verify that $(H\bigcup\{e_0,\overline{e_0}\},B)$ makes a multiplicative pair.
According to the previous arguments, it remains to verify equation \eqref{eq:2delta11} for all $y\in B_1$ by symmetry.
\par
We claim that for all $y \in B_1$ if $e_i+y\in B$ then $e_i+y\in B_1$.
With this, it is clear that equation \eqref{eq:2delta11} holds for all $y\in B_1$ and the theorem follows.
\par
The claim is confirmed in two cases according to the weight of $y.$
First, if $|y|\geq \frac{n+1}{2}+1$ then $|e_i+y|\geq \frac{n+1}{2}$.
Certainly, if $e_i+y\in B$ then $e_i+y\in B_1$.
Next, if $|y|=\frac{n+1}{2}$ and $y_1=0$, then $|e_i+y|$ is either $\frac{n+1}{2}+1$ or $\frac{n-1}{2}$.
If $|e_i+y| = \frac{n+1}{2}+1$, then $e_i+y\in B$ implies $e_i+y\in B_1$.
If $|e_i+y|$ were $\frac{n-1}{2}$, then $y_i=1.$
Therefore, if $e_i+y\in B$ then $e_i+y\in \overline{e_0}+B_1$ by the definition of $B.$
More precisely, $e_i+y \in \{x\in \Z_2^n \mid |x|=\frac{n-1}{2}\textrm{ and } x_1=1\},$ and thus $(e_i+y)_1=1.$
But this means that $i=1$ since $y_1=0$ by assumption, which contradicts with the condition $y_i=y_1=1$ imposed earlier on $y.$
\end{proof}
\subsection{The case in which $n\equiv 2\mod 4$}\label{subsec:Yuzvinsky2}
The idea is similar, only in this case more complicated $\Z_2^n$-graded quasialgebras
and more delicate perturbation functions get involved.
We shall consider the algebra $\bbP_n(n)$ and its H-set $H=\{e_i, e_i+e_n \mid 1\leq i\leq n\}.$
The verification that $H$ is indeed an H-set is fairly easy and thus omitted.
Note that $\r(2^n)=2n$ and we need to add two elements into the chosen H-set.
We intend to expand $H$ to $H\bigcup\{\overline{e_0}, \overline{e_n}\}$
according to the apparent symmetry of $H$ with respect to $e_n.$
Let $f=f_{\bbP(n)}$ and perturb it by $\delta_1$, $\delta_2$ and $\delta_3$ in the following way:
\begin{equation*}
f'(x,y)=\left\{\begin{array}{ll}
                f(x,y)+\delta_1(y), & \textrm{ if } x = \overline{e_0}; \\
                f(x,y)+\delta_2(y), & \textrm{ if } x = \overline{e_n}; \\
                f(x,y)+\delta_3(y), & \textrm{ if } x = e_0; \\
                f(x,y),             & \textrm{ otherwise. }
               \end{array}\right.
\end{equation*}
\par
Similarly, by routine calculations we have the following
\begin{lem}
$(H\bigcup\{\overline{e_0},\overline{e_n}\}, B)$ forms a multiplicative pair with respect to $f'$
if and only if for $\forall h\in H\setminus\{e_0\}$ the following equations hold:
\begin{align*}
& f'(\overline{e_0},y)+f'(\overline{e_0},y+\overline{h})=f(h,\overline{h})+1, \quad
  \textrm{ whenever } y, y+\overline{h}\in B; \\
& f'(\overline{e_n},y)+f'(\overline{e_n},y+\overline{h+e_n})=f(h,h+\overline{e_n})+1, \quad
  \textrm{ whenever } y, y+\overline{h+e_n}\in B; \\
& f'(e_0,y)+f'(e_0,y+h)=f(h,h)+1, \quad
  \textrm{ whenever } y, y+h\in B; \\
& f'(e_0,y)+f'(e_0,\overline{y})+f'(\overline{e_0},y)+f'(\overline{e_0},\overline{y})=1, \quad
  \textrm{ whenever } y, \overline{y}\in B; \\
& f'(e_0,y)+f'(e_0,y+\overline{e_n})+f'(\overline{e_n},y)+f'(\overline{e_n},y+\overline{e_n})=1, \quad
  \textrm{ whenever } y, y+\overline{e_n}\in B; \\
& f'(\overline{e_0},y)+f'(\overline{e_0},e_n+y)+f'(\overline{e_n},y)+f'(\overline{e_n},e_n+y)=1, \quad
  \textrm{ whenever } y, e_n+y\in B,
\end{align*}
or equivalently
\begin{align}
\label{eq:delta21}
& \delta_1(y)+\delta_1(y+\overline{h})=f(\overline{e_0},h)+f(h,\overline{e_0})+1, \quad
  \textrm{ whenever } y, y+\overline{h}\in B; \\
\label{eq:delta22}
& \delta_2(y)+\delta_2(y+\overline{h}+e_n)=f(\overline{e_n},h)+f(h,\overline{e_n})+1, \quad
  \textrm{ whenever } y, y+\overline{h+e_n}\in B; \\
\label{eq:delta23}
& \delta_3(y)+\delta_3(y+h)=f(h,h)+1, \quad
  \textrm{ whenever } y, y+h\in B; \\
\label{eq:delta24}
& \delta_3(y)+\delta_3(\overline{y})+\delta_1(y)+\delta_1(\overline{y})=0, \quad
  \textrm{ whenever } y, \overline{y}\in B; \\
\label{eq:delta25}
& \delta_3(y)+\delta_3(y+\overline{e_n})+\delta_2(y)+\delta_2(y+\overline{e_n})=0, \quad
  \textrm{ whenever } y, y+\overline{e_n}\in B; \\
\label{eq:delta26}
& \delta_1(y)+\delta_1(e_n+y)+\delta_2(y)+\delta_2(e_n+y)=0, \quad
  \textrm{ whenever } y, y+e_n\in B.
\end{align}
\end{lem}
\par
Again, if we set $\delta=\delta_1=\delta_2=\delta_3$,
then equations \eqref{eq:delta24}, \eqref{eq:delta25} and \eqref{eq:delta26} hold automatically.
With this reduction, one has
\begin{lem}
$(H\bigcup\{\overline{e_0},\overline{e_n}\},B)$ is a multiplicative pair
if and only if the following equations hold:
\begin{align}
\label{eq:7delta21}
& \delta(y)+\delta(y+\overline{e_i})=1,     \quad \forall i\neq n, \textrm{ whenever } y, y+\overline{e_i}\in B; \\
\label{eq:7delta22}
& \delta(y)+\delta(y+\overline{e_i}+e_n)=1, \quad \forall i\neq n, \textrm{ whenever } y, e_n+y+\overline{e_i}\in B; \\
\label{eq:7delta23}
& \delta(y)+\delta(y+e_i)=0,                \quad \forall i\neq n, \textrm{ whenever } y, y+e_i\in B; \\
\label{eq:7delta24}
& \delta(y)+\delta(y+e_i+e_n)=0,            \quad \forall i\neq n, \textrm{ whenever } y, e_n+y+e_i\in B; \\
\label{eq:7delta25}
& \delta(y)+\delta(y+\overline{e_n})=1,     \quad \textrm{ whenever } y, y+\overline{e_n}\in B; \\
\label{eq:7delta26}
& \delta(y)+\delta(y+\overline{e_n}+e_n)=1, \quad \textrm{ whenever } y, e_n+y+\overline{e_n}\in B; \\
\label{eq:7delta27}
& \delta(y)+\delta(y+e_n)=0,                \quad \textrm{ whenever } y, e_n+y\in B.
\end{align}
\end{lem}
\begin{proof}
In equations \eqref{eq:delta21}-\eqref{eq:delta23} let $h$ take turns to be $e_n$, $e_i$ and $e_i+e_n$, $i \ne n$ and carry out direct calculations.
\end{proof}
\par
We proceed to decide $B$ in light of the preceding lemma and simultaneously by taking advantage of the obvious symmetry of $\Z_2^n$ as before.
Assume $B=B_0 \biguplus e_n+B_0$ in which $B_0$ is a subset of $\{y\in \Z_2^n \mid y_n=0\}$.
Assume also that for all $y\in \Z_2^n,$ the equation $\delta(y)+\delta(y+e_n)=0$ holds.
Together with the symmetry of $B,$ equations \eqref{eq:7delta21}-\eqref{eq:7delta27} are further reduced to:
\begin{align}
\label{eq:3delta21}
& \delta(y)+\delta(y+e_i)=0,
  \quad \forall i\neq n, \textrm{ whenever } y\in B_0, y+e_i\in B_0; \\
\label{eq:3delta22}
& \delta(y)+\delta(y+e_i+\overline{e_n})=1,
  \quad \forall i\neq n, \textrm{ whenever } y\in B_0, y+e_i+\overline{e_n}\in B_0; \\
\label{eq:3delta23}
& \delta(y)+\delta(y+\overline{e_n})=1,
  \quad                  \textrm{ whenever } y\in B_0, y+\overline{e_n}\in B_0.
\end{align}
\par
\begin{thm}
\label{thm:Yuz2}
There exists a series of admissible triples of sizes
$$\left[ 2n+2, 2^n-4\binom{n-2}{\dfrac{n-2}{2}}, 2^n\right],
  \quad\textrm{ for } n\geq 4 \textrm{ and } n\equiv 2 \mod 4. $$
\end{thm}
\begin{proof}
We identify $\Z_2^{n-1}$ with $\{y\in \Z_2^n \mid y_n=0\}$
by sending vector $y\in\Z_2^{n-1}$ to vector $(y,0)\in\Z_2^n$.
Thus, $\overline{e_n}$ functions exactly as $\overline{e_0}$ in $\Z_2^{n-1}$.
From this point of view, equations \eqref{eq:3delta21}-\eqref{eq:3delta23}
are essentially identical to equations \eqref{eq:2delta11}-\eqref{eq:2delta13}.
Hence we can choose $\delta$ and $B$ in a similar manner. More precisely, let
\[B_0=B'_1 \biguplus \overline{e_n}+B'_1, \]
where $B'_1=\{(y,0) \mid y\in B_1\}$ and
$$B_1=\{y\in \Z_2^{n-1}\mid |y|\geq\frac{n}{2}+1\}\biguplus \{y\in\Z_2^{n-1}\mid |y|=\frac{n}{2}\textrm{ and } y_1=0\};$$
let $$\delta(y)=\left\{\begin{array}{ll}
                        1, & \textrm{ if } y\in B'_1\biguplus e_n+B'_1;\\
                        0, & \textrm{ otherwise. }
                       \end{array}\right.$$
Hence $B = B'_1\biguplus\{\overline{e_n}+B'_1\}\biguplus \{e_n+B'_1\}\biguplus\{\overline{e_0}+B'_1\},$
and by simple counting $$\card B=2^n-4\binom{n-2}{\dfrac{n-2}{2}}.$$
Now with the easy fact $H\bigcup\{\overline{e_0}, \overline{e_n}\} + B=\Z_2^n$ the theorem follows.
\end{proof}
\section{Improvement of the Lenzhen-Morier-Genoud-Ovsienko Formulas}\label{sec:LMGO}
In principle, given any admissible triple $[r,s,N],$ one can also construct new triples from it
by decreasing the second and the third entries while preserving the first.
Similar to the method of Yuzvinsky,
the construction relies on an explicit realization of the known triple by multiplicative pairs.
Morier-Genoud and Ovsienko constructed the Hurwitzian sets in their higher octonions $\bbO_n$
with the exception $n \equiv 0 \mod 4$ in \cite{MGO},
and this enabled them to apply this idea on the Hurwitz-Radon triple $[\r(2^n),2^n,2^n]$
to provide two infinite families of admissible triples in their joint work \cite{LMGO} with Lenzhen.
\par
Note that in our series of algebras $\bbP_n(m)$ the Hurwitzian sets (of cardinality $\r(2^n)$)
can be constructed for all $n.$ Hence we are able to increase the first entry of
the Lenzhen-Morier-Genoud-Ovsienko triples in the case of $n\equiv 0\mod 4.$
In addition, we observe that there is a flaw in their construction and by a proper remedy
we can strengthen their admissible triples in some cases.
\par
\subsection{The idea of constructing new triples}
To begin with, we give more details on the idea of constructing new triples from the known ones.
Suppose that $(A,B)$ is a multiplicative pair in a suitable $\Z_2^n$-graded quasialgebra.
Obviously, if $A_1\subseteq A$ and $B_1\subseteq B,$ then $(A_1,B_1)$ is a multiplicative pair as well and
it results in a new triple $[\card{}(A_1),\,\card{}(B_1),\,\card{}(A_1+B_1)]$ by Corollary \ref{cor:multieq}.
\par
There are two basic methods of constructing new triples in which the first entries are preserved.
One is called \emph{the method of addition}.
Let $\{B_i\}$ be the set of all subsets of $B$.
Then we get a series of new triples $[\card{}(A),\,\card{}(B_i),\,\card{}(A+B_i)]$.
The subtlety is that one may move a step forward.
Namely, by noticing that there might exist extra elements $y \in B\setminus B_i$ such that $A+\{y\}\subseteq A+B_i$
and if this is the case, we can expand $B_i$ by $\{y\}$ and get better triples.
Dually, the other is \emph{the method of subtraction}.
Let $C=A+B$ and let $\{C_i\}$ be the set of all subsets of $C$.
For each $C_i,$ one has the new triple
$[\card(A),\card(B \setminus (A+C_i)),\card(A+\{B\setminus (A+C_i)\})].$
The reason why the last entry is not $\card(C\setminus C_i)$ is that
the removal of $A+C_i$ from $B$ may lead to the removal of some extra elements in $C\setminus C_i.$
So one should use $\card(A+\{B\setminus (A+C_i)\})$
instead of $\card(C\setminus C_i)$ to get the more accurate formula.
\par
In the following we give a description of the condition about
whether the subtlety may happen in the multiplicative pair $(H, \Z_2^n).$
\begin{prop}
\label{prop:DualRelation}
Let $H$ be a Hurwitzian set and let $B'\subseteq\Z_2^n$.
Then $H+\{\Z_2^n\setminus (H+B')\}=\Z_2^n\setminus B'$ if and only if
for $\forall\ z \in \Z_2^n\setminus B', \ H+z\nsubseteq H+B'$,
\end{prop}
\begin{proof}
Clearly, the relation $H+\{\Z_2^n\setminus (H+B')\} \varsubsetneq \Z_2^n\setminus B'$ holds
if and only if there exists an element $z \in \Z_2^n\setminus B'$,
such that $\forall\ y\in \Z_2^n\setminus (H+B')$,
$z \notin H+y$, or equivalently $y \notin H+z$.
Hence one has $H+z\subseteq H+B'$. This completes the proof.
\end{proof}
\begin{rem} \label{r5.2}
There is a dual version for the method of addition.
In addition, if the condition of the proposition holds then the third entry of the triple
$[\card{}(H),2^n-\card{}(H+B'),2^n-\card{}(B')]$ will not decrease by removing extra elements in $\Z_2^n \setminus B',$
and in the meantime, the second entry of the triple  $[\card{}(H),\card{}(B'),\card{}(H+B')]$
will not increase by adding extra elements to $B'.$
\end{rem}
\begin{ex}
From Subsection \ref{subsec:anotherhset},
one has $H=\{e_0, e_i, \overline{e_i} \mid 1\leq i\leq 4\}$ is an H-set in case $n=4$ and $H+\Z_2^4=\Z_2^4.$
Using the method of subtraction, to remove $B'=\{e_0,e_1\}$ from the second $\Z_2^4,$
one has to remove $H+B'$ from the first $\Z_2^4$.
However, in this situation one easily derives that $H+B'=\Z_2^4$.
Therefore, one has such an unusual formula:
\[ H+\{\Z_2^4\setminus (H+B')\} = H+\varnothing = \varnothing = \varsubsetneq \Z_2^4\setminus B'.\]
Dually, after using the method of addition to get $[\card{}(H),\card{}(B'),\card{}(\Z_2^4)]$
one needs to move forward by expanding $B'$ to $Z_2^4$
in order to get a better triple $[\card{}(H),\card{}(\Z_2^4),\card{}(\Z_2^4)].$
This example explains the subtlety in both the method of subtraction and the method of addition.
\end{ex}
\subsection{Strengthening the Lenzhen-Morier-Genoud-Ovsienko formulas}
Take $H$ to be the Hurwitzian set defined in \eqref{eq:anotherhset} and let $B$ be a subset of $\Z_2^n.$
Expand $B$ to the utmost $B'$ such that $H+B=H+B'$
and obtain triples $[\card{}(H),\card{}(B'),\card{}(H+B')]$ and $[\card{}(H),2^n-\card{}(H+B'),2^n-\card{}(B')]$
by Proposition \ref{prop:DualRelation}.
\par
For convenience, let $e_{ij}\triangleq e_i+e_j$ and $\overline{e_{ij}}\triangleq \overline{e_0}+e_{ij}$
and define $\Gamma(l,k)$ with $1\leq l< k\leq n$ as
$$\Gamma(l,k)=\{e_{ij}\in\Z_2^n \mid 1\leq i<j<k\}\bigcup \{e_{ik}\in\Z_2^n \mid 1\leq i\leq l\}.$$
\par
We will construct subsets $B_{lk}$ and $B'_{lk}$ via $\Gamma(l,k)$ and get the following triples:
\begin{thm}\label{thm:firstfamilies}
Let $n > 4$ and $1\leq l<k\leq n.$
\begin{enumerate}
\item If $n\geq 5$ and $n=1,2 \mod 4$ then there exist square identities of sizes $[r,s,N]$ with
\begin{equation}
\left\{\begin{array}{rcl}
r & = & \r(2^n),                                             \\
s & = & 2\left(\binom{k-1}{2}+l+1\right),                    \\
N & = & 2\left(\binom{k-1}{2}+l+1\right)n -4\binom{k}{3}-2kl,
\end{array}\right.
\end{equation}
while $k\leq n-2$; \\
if $k=n-1$ and $l\leq n-3$, there are square identities of sizes $[r,s+2l,N]$; \\
if $k=n-1$ and $l=n-2$, there are square identities of sizes $[r,s+2l+2,N]$;   \\
if $k=n$, there are square identities of sizes $[r,s+2(n-1-l),N]$.
\item If $n\ge 8$ and $n=0,3 \mod 4$ then there exist square identities of sizes $[r,s,N]$ with
\begin{equation}
\left\{\begin{array}{rcl}
r&=&\r(2^n),                                                \\
s&=&2\left(\binom{k-1}{2}+l+1\right),                       \\
N&=&2\left(\binom{k-1}{2}+l+1\right)(n+1)-4\binom{k}{3}-2kl,
\end{array}\right.
\end{equation}
while $k\leq n-1$; \\
if $k=n$ and $l\leq n-2$, there are square identities of sizes $[r,s+2l,N]$; \\
if $k=n$ and $l=n-1$, there are square identities of sizes $[r,s+2l+2,N]$.
\end{enumerate}
\end{thm}
\begin{proof}
In this proof, we always assume $i<j$ in the notation of $e_{ij}$.
\begin{enumerate}
\item {\em Cases in which $n \equiv 1,2 \mod 4$}. Recall that $H=\{e_m,e_1+e_m \mid 1\leq m\leq n\}$ is an H-set.
Let $B_{lk} = \{e_0,e_1, e_{ij}, e_1+e_{ij} \mid e_{ij}\in\Gamma(l,k)\}$
whose cardinality is $$\card(B_{lk})=2+2\binom{k-1}{2}+2l.$$
The summation of $H$ and any element $e_{ij}$ of $\Gamma(l,k)$ is
$$H+e_{ij}=\{e_i,e_j,e_{ij}+e_m,e_1+e_i,e_1+e_j,e_1+e_{ij}+e_m \mid i\neq m\neq j\},$$
which is identical to $H+\{e_1+e_{ij}\}$ as $H$ is symmetric with respect to $e_1.$
Hence the corresponding sumset is then of the form
$$H+B_{lk}=H\cup \bigcup_{e_{ij}\in\Gamma(l,k)}\{H+e_{ij}\}.$$
Therefore its cardinality is
\begin{align*}
\card(H+B_{lk}) & = 2n+2\left(\sum_{1\leq i<j \leq k-1}(n-j)+l(n-k)\right)\\
                & = 2\left(\binom{k-1}{2}+l+1\right)n-4\binom{k}{3}-2kl.
\end{align*}
\par
Further, we expand $B_{lk}$ to its utmost $B'_{lk}$
by adding elements in $\Z_2^n\setminus B_{lk}$ by enumeration.
Suppose $e_x$ is an element not in $B_{lk}$.
\begin{enumerate}
\item If $|e_x|\geq 5$, then $e_x+e_0$ is not in $H+B_{lk}.$
Hence no elements of this form can be added into $B'_{lk}.$
\item If $e_x=e_{i_1}+e_{i_2}+e_{i_3}+e_{i_4}$ with $i_1<i_2<i_3<i_4,$ then there exists an element $e_{i_5}\in H$
where $i_5\notin \{i_1,i_2,i_3,i_4\}$, such that $e_x+e_{i_5} \notin H+B_{lk}.$
Therefore $B'_{lk}$ contains no elements of this form.
\item If $e_x=e_{i_1}+e_{i_2}+e_{i_3}$ with $1<i_1<i_2<i_3$ then there exists an element $e_{i_4}\in H$
where $i_4\notin \{1,i_1,i_2,i_3\}$, such that $e_x+e_{i_4} \notin H+B_{lk}$.
Hence $e_x$ can not be added into $B'_{lk}.$
If $e_x=e_1+e_{i_1 i_2}$ where $e_{i_1 i_2}\notin\Gamma(l,k)$ and $1<i_1<i_2$,
then we have $e_x+H=e_{i_1 i_2}+H$.
The next case is similar to this, so we discuss them together in the following.
\item If $e_x=e_{i_1 i_2}$ with $e_{i_1 i_2}\notin\Gamma(l,k)$,
then whether $e_x$ satisfies the condition of Proposition \ref{prop:DualRelation}
depends on the parameters $l$ and $k$.
And as mentioned above, we only need to check whether $e_{i_1 i_2}+e_m \in H+B_{lk}$.
Since $e_{i_1 i_2}+e_m$ can only appear in $H+\{e_{i_1 m}\}$ or $H+\{e_{i_2 m}\}$,
it suffices to check whether $e_{i_1 m}$ or $e_{i_2 m}$ belongs to $\Gamma(l,k)$ for all $m\neq i_1, i_2$.
If so, then $e_x$ and $e_1+e_x$ can be added into $B'_{lk}$.
In the following we separate the discussions into several cases:
\par
\begin{enumerate}
\item $k\leq n-2$.
\begin{enumerate}
\item If $i_1<i_2\leq n-1$, then $e_{i_1 m},\ e_{i_2 m}\notin \Gamma(l,k)$ when $m=n.$
\item If $i_1<n-1$ and $i_2=n$, then $e_{i_1 m},\ e_{i_2 m}\notin \Gamma(l,k)$ when $m=n-1.$
\item If $i_1=n-1$ and $i_2=n$, then $e_{i_1 m},\ e_{i_2 m}\notin \Gamma(l,k)$ when $m=n-2.$
\end{enumerate}
Hence in this case, no elements can be added into $B'_{lk}.$
\item $k=n-1$ and $l\leq n-3$.
\begin{enumerate}
\item If $i_1\leq l$ and $i_2=n$,
then $e_{i_1 m} \in\Gamma(l,k)$ holds for all $m\neq i_1, i_2,$
so these $2l$ elements of the form $e_{i_1 i_2}$ and $e_1+e_{i_1 i_2}$ can be added into $B'_{lk}.$
\item If $l<i_1<n-1$ and $i_2=n$, then $e_{i_1 m},\ e_{i_2 m}\notin \Gamma(l,k)$ when $m=n-1.$
\item If $i_1=n-1$ and $i_2=n$, then $e_{i_1 m},\ e_{i_2 m}\notin \Gamma(l,k)$ when $m=n-2.$
\item If $l<i_1<i_2=n-1$, then $e_{i_1 m},\ e_{i_2 m}\notin \Gamma(l,k)$ when $m=n.$
\end{enumerate}
\item $k=n-1$ and $l=n-2$.
\begin{enumerate}
\item If $i_1\leq l$ and $i_2=n$,
then $e_{i_1 m} \in\Gamma(l,k)$ holds for all $m\neq i_1, i_2$,
so these $2l$ elements of the form $e_{i_1 i_2}$ and $e_1+e_{i_1 i_2}$ can be added.
\item If $i_1=n-1$ and $i_2=n$, then $e_{i_1 m}$ is always in $\Gamma(l,k)$ for all $m\neq i_1, i_2$,
so these $2$ elements $e_{i_1 i_2}$ and $e_1+e_{i_1 i_2}$ can be added.
\end{enumerate}
\item $k=n$. \\
If $l<i_1<i_2=n$,
then $e_{i_1 m} \in\Gamma(l,k)$ holds for all $m\neq i_1, i_2$,
so these $2(n-1-l)$ elements of the form $e_{i_1 i_2}$ and $e_1+e_{i_1 i_2}$ can be added.
\end{enumerate}
\item If $e_x=e_{i_1}\notin B_{lk},$ then $H+e_{i_1}$ is exactly $H+e_{1 i_1}.$
This has already been discussed in the above case.
\end{enumerate}
\item The proof for the cases in which $n \equiv 0,3 \mod 4$ is similar to the above cases
except that $\overline{e_0}$ plays the role of $e_1$.
\end{enumerate}
\end{proof}
\par
From the proof of the above theorem, one can easily conclude that
$B'_{lk}$ is the utmost subset that can be expanded from $B_{lk}$.
Hence by Proposition \ref{prop:DualRelation} and Remark \ref{r5.2}, one has:
\begin{cor} Let $[r,s',N]$ be the admissible triples obtained in Theorem \ref{thm:firstfamilies}.
Then there exists a series of admissible triples of the form $[r,2^n-N,2^n-s'].$
\end{cor}
\begin{ex}
Theorem \ref{thm:firstfamilies} generates several best known admissible triples, see \cite{S}.
\begin{equation*}
\setlength{\extrarowheight}{2pt}
\begin{array}{c|c|c|c}  \hline
n & (l,k) & \textrm{triple} & \textrm{method}         \\ \hline  
5 & (2,4) & [10,16,28]      & \textrm{addition}       \\ \hline  
5 & (3,4) & [10,22,30]      & \textrm{addition}       \\ \hline  
6 & (1,2) & [12,44,60]      & \textrm{subtraction}    \\ \hline  
6 & (1,3) & [12,38,58]      & \textrm{subtraction}    \\ \hline  
6 & (4,5) & [12,32,52]      & \textrm{addition}       \\ \hline  
\end{array}
\end{equation*}
\end{ex}
\begin{rem}
The admissible triples obtained here strengthen the Lenzhen-Morier-Genoud-Ovsienko formulas in the following aspects:
Firstly, since our twisting function is a bit more complex than theirs,
we are able to avoid the exception of finding Hurwitzian sets when $n\equiv 0\mod 4$
and optimise the corresponding triples;
Secondly, by observing the subtlety of the method of subtraction in constructing new triples,
the Lenzhen-Morier-Genoud-Ovsienko formulas can be improved in several cases,
especially when $k$ is close to $n$ the third entries of their triples can be decreased;
Thirdly, using the same method of constructing the second families of 
the Lenzhen-Morier-Genoud-Ovsienko triples \cite[Theorem 2]{LMGO},
we can also provide similar families of triples for $n\equiv 0, 2\mod 4$. 
\end{rem}
\par
Furthermore, by observing the symmetry of the Hurwitzian set in the case $n\ge 8$ and $n=0\mod 4,$
we are able to expand the corresponding multiplicative pairs and improve the triples:
\begin{prop}
If $n\ge 8$ and $n=0 \mod 4,$ then there exist square identities of sizes $[r,s,N]$ with
\begin{equation}
\left\{\begin{array}{rcl}
r&=&\r(2^n)+1,                                                \\
s&=&2\left(\binom{k-1}{2}+l+1\right),                       \\
N&=&2\left(\binom{k-1}{2}+l+1\right)(n+1)-4\binom{k}{3}-2kl,
\end{array}\right.
\end{equation}
while $k\leq n-1$; \\
if $k=n$ and $l\leq n-2$, there are square identities of sizes $[r,s+2l,N]$; \\
if $k=n$ and $l=n-1$, there are square identities of sizes $[r,s+2l+2,N]$.
\end{prop}
\begin{proof}
In fact, notice that in this case we expand the Hurwitzian set $H$ by $\overline{e_0}$ to $H',$
thus obtain a more symmetric multiplicative pair $(H',B)$,
which yields the first family of the Yuzvinsky-Lam-Smith formulas.
By applying the method of addition on this pair and expanding the subset
$B_{lk} = \{e_0,\overline{e_0}, e_{ij}, \overline{e_0}+e_{ij} \mid e_{ij}\in\Gamma(l,k)\}$ to $B'_{lk}$
in the same manner, one obtains the desired result.
The reason why this procedure works is that
$B'_{lk}$ contains elements whose weights are either $\le 2$ or $\ge n-2$, i.e.,
$B'_{lk}$ does not intersect with $\Z_2^n\setminus B.$
\end{proof}
\subsection{The cases in which $n=4$ and $n=7$}
Notice that in Theorem \ref{thm:firstfamilies}, we exclude the cases in which $n=4$ and $7.$
The reason is that the general argument therein breaks down on these two specific cases.
Nevertheless, with care we can find suitable multiplicative pairs and realize the associated triples.
For the sake of completeness, we list in below the results while omitting further details:
\begin{enumerate}
\item If $n=4$ then we have the admissible triples $[9,s,8+s]$ for $s=1,2,\cdots,7.$
These are very well-known triples, see \cite{S}.
\item If $n=7$ then the formulas in Theorem \ref{thm:firstfamilies} are still correct for $k\leq 5$ and for $l=1, \ k=6.$
As for the others, we list them in the following table:
\begin{equation*}
\setlength{\extrarowheight}{2pt}
\begin{array}{c|c|c||c|c|c}  \hline
(l,k) & \textrm{addition} & \textrm{subtraction} & (l,k) & \textrm{addition} & \textrm{subtraction} \\ \hline
(2,6) & [16,30,104]       & [16,24,98]           & (3,6) & [16,40,108]       & [16,20,88]           \\ \hline
(4,6) & [16,54,112]       & [16,16,74]           & (5,6) & [16,72,116]       & [16,12,56]           \\ \hline
(1,7) & [16,76,118]       & [16,10,52]           & (2,7) & [16,82,120]       & [16,8,46]            \\ \hline
(3,7) & [16,90,122]       & [16,6,38]            & (4,7) & [16,100,124]      & [16,4,28]            \\ \hline
(5,7) & [16,112,126]      & [16,2,16]            & (6,7) & [16,128,128]      & [16,0,0]             \\ \hline
\end{array}
\end{equation*}
\end{enumerate}
\section{Some New Families of Admissible Triples}\label{sec:newSeries}
In this section, we apply the method of addition on the Yuzvinsky formulas
and obtain some new families of admissible triples: 
\begin{thm}
\label{thm:newSeries}
Let $n\ge 4.$
\begin{enumerate}
\item If $n\equiv 0\mod 4$, then there exists a series of admissible triples of sizes:
\[\left[ 2n+2, 2\sum_{i=0}^{k}\binom{n}{i}, 2\sum_{i=0}^{k+1}\binom{n}{i}\right],\quad k=0,\cdots,\frac{n}{2}-2.\]
\item If $n\equiv 1\mod 4$, then there exists a series of admissible triples of sizes:
\[\left[ 2n+2, 2\sum_{i=0}^{k}\binom{n}{i}, 2\sum_{i=0}^{k+1}\binom{n}{i}\right],\quad k=0,\cdots,\frac{n-1}{2}-2.\]
\item If $n\equiv 2\mod 4$, then there exists a series of admissible triples of sizes:
\[\left[ 2n+2, 4\sum_{i=0}^{k}\binom{n-1}{i}, 4\sum_{i=0}^{k+1}\binom{n-1}{i}\right],\quad k=0,\cdots,\frac{n-2}{2}-2.\]
\end{enumerate}
\end{thm}
\begin{proof}
Denote by $(H',B)$ the multiplicative pairs corresponding to the Yuzvinsky admissible triples in each case.
We prove the theorem case by case according to the residue of $n$ modulo $4.$
\begin{enumerate}
\item $n\equiv 0\mod 4$.
Recall from Subsection \ref{subsec:Yuzvinsky0} that
$$H'=\{e_i, \overline{e_i} \mid 0\leq i\leq n\}\textrm{ and } B=\{y\in \Z_2^n \mid |y|\ne n/2\}.$$
Let $B_k=\{y\in\Z_2^n \mid |y|\leq k \hbox{ or } |y|\geq n-k\},\ k=0,\cdots,n/2-2.$
Then one has $B_k\subseteq B$ and $\card(B_k) = 2\sum_{i=0}^{k}\binom{n}{i}.$
By the symmetry of $H'$ and $B_k$, it is not hard to see that
$$H'+B_k=B_{k+1}=\{y\in\Z_2^n \mid |y|\leq k+1 \hbox{ or } |y|\geq n-k-1\}.$$
Moreover, for any $y\in B\setminus B_k$, one has $H'+y\nsubseteq B_{k+1}.$
\item $n\equiv 1\mod 4$.
Recall from Subsection \ref{subsec:Yuzvinsky1} that $H'=\{e_i, \overline{e_i} \mid 0\leq i\leq n\}$
and $$B=\Z_2^n\setminus\left\{\{y\in\Z_2^n\mid |y|=\frac{n+1}{2}\textrm{ and } y_1=1\}
\biguplus \{y\in\Z_2^n\mid |y|=\frac{n-1}{2}\textrm{ and } y_1=0\}\right\}.$$
Let $B_k=\{y\in\Z_2^n \mid |y|\leq k \hbox{ or } |y|\geq n-k\},\ k=0,\cdots,(n-1)/2-2.$
Then by a similar manner, one obtains the desired result.
\item $n\equiv 2\mod 4$.
Recall from Subsection \ref{subsec:Yuzvinsky2} that
$$H'=\{e_i, e_i+e_n, \overline{e_0}, \overline{e_n} \mid 1\leq i\leq n\}$$
and $B = B'_1\biguplus\{\overline{e_n}+B'_1\}\biguplus \{e_n+B'_1\}\biguplus\{\overline{e_0}+B'_1\},$
where $B'_1=\{(y,0) \mid y\in B_1\}$ and
\begin{equation*}
B_1=\{y\in \Z_2^{n-1}\mid |y|\geq\frac{n}{2}+1\}\biguplus\{y\in\Z_2^{n-1}\mid |y|=\frac{n}{2}\textrm{ and } y_1=0\}.
\end{equation*}
For $k=0,\cdots,(n-2)/2-2,$ let
$B_k = B'_{1k}\biguplus\{\overline{e_n}+B'_{1k}\}\biguplus \{e_n+B'_{1k}\}\biguplus\{\overline{e_0}+B'_{1k}\},$
where $B'_{1k}=\{(y,0) \mid y\in B_{1k}\}$ and
$B_{1k}=\{y\in\Z_2^{n-1} \mid |y|\leq k \hbox{ or } |y|\geq n-1-k\}.$
It is obvious that $B_k\subseteq B$ and $\card(B_k) = 4\sum_{i=0}^{k}\binom{n-1}{i}.$
Similarly, by the symmetry of $H'$ and $B_k$, one has $H'+B_k=B_{k+1}.$
Moreover, if $y\in B\setminus B_k$ then $H'+y\nsubseteq B_{k+1},$
which means that $B_k$ can not be expanded in this situation.
\end{enumerate}
\end{proof}
\begin{rem}
Note that if increasing $k$ by $1$ at the ``border'' of each case in the above theorem,
then one will recover the Yuzvinsky admissible triples by expanding $B_k$ to $B.$
Take $k=1$ in the first two series of admissible triples, 
then the resulting square identities improve the well-known Lagrange identity. 
More preciously, our third entries are nearly half of that of the latter. 
In general, our new triples look very similar to those of \cite[Theorem 2]{LMGO} 
as the idea of construction is nearly the same. 
The only difference is that our starting point is the Yuzvinsky triples 
instead of the Hurwitz-Radon triples in {\it loc. cit.}. 
\end{rem}

\end{document}